\documentclass[11pt]{article}
\usepackage{amssymb,amsmath,amsthm}
\usepackage[colorlinks=true,breaklinks]{hyperref}
\usepackage{xcolor}
\definecolor{c1}{rgb}{0,0,1} 
\definecolor{c2}{rgb}{0,0.3,0.9} 
\definecolor{c3}{rgb}{0.3,0,0.7} 
\hypersetup{
    linkcolor={c1}, 
    urlcolor={c2}, 
    citecolor={c3} 
}
\newcommand{\R}{\mathbb{R}}

\newcommand{\N}{\mathbb{N}}

\newtheorem{theorem}{Theorem}[section]
\newtheorem{corollary}[theorem]{Corollary}
\newtheorem{lemma}[theorem]{Lemma}

\theoremstyle{definition}
\newtheorem{definition}[theorem]{Definition}

\theoremstyle{remark}
\newtheorem{remark}[theorem]{Remark}
\numberwithin{equation}{section}

\DeclareMathOperator{\meas}{meas}

\begin{document}


\title{Multiplicity results for a fractional Schr\"{o}dinger equation with potentials}
\date{}
\date{}
\author{Sofiane Khoutir\thanks{{E-mail: sofiane\_math@live.fr}}\\
{\small Faculty of Mathematics, University of Science and Technology Houari Boumediene,}\\
{\small PB 32 El Alia, Bab Ezzouar 16111, Algiers, Algeria}}
\maketitle
 \begin{center}
 \begin{minipage}{13cm}
 \par
   \small  {\bf Abstract:} This paper is devoted to study a class of nonlinear fractional Schr\"{o}dinger equations:
\begin{equation*}
(-\Delta)^{s}u+V(x)u=f(x,u), \quad \text{in}\: \R^{N},
\end{equation*}
where $s\in (0,1)$, $\ N>2s$, $(-\Delta)^{s}$ stands for the
fractional Laplacian. The main purpose of this paper is to study the existence of infinitely many solutions for the aforementioned equation. First, by using a variational approach, we establish the existence of at least one nontrivial solution for the above equation with a general potential $V(x)$ which is allowed to be sign-changing and a sublinear nonlinearity $f(x,u)$. Next, by using variational methods and the Moser iteration technique, we prove the existence of infinitely many solutions with $V(x)$ is a nonnegative potential  and the nonlinearity $f(x,u)$ is locally sublinear with respect to $u$.
\end{minipage}
\end{center}

 \vskip6mm
\vskip2mm
 \par
\section{Introduction and Main results}
Consider the following fractional Schr\"{o}dinger equations:
\begin{equation}
(-\Delta )^{s}u+ V(x)u=f(x,u), \quad \text{in}\: \R^{N}, 
\label{1}
\end{equation}
where $s \in (0,1)$, $N>2s$, $(-\Delta )^{s}$ stands for the fractional Laplacian.

The fractional Schr\"{o}dinger equation is a fundamental equation of fractional quantum mechanics. It was discovered by Laskin \cite{1,2} as a result of extending the Feynman path integral, from the Brownian-like to L\'{e}vy-like quantum mechanical paths, where the Feynman path integral leads to the classical Schr\"{o}dinger equation, and the path integral over L\'{e}vy trajectories leads to the fractional Schr\"{o}dinger equation.

The equation \eqref{1} with $s=1$ is the nonlinear Schr\"{o}dinger equation
\begin{equation}\label{3}
-\Delta u+ V(x)u=f(x,u), \quad \text{in}\: \R^N,
\end{equation}
which has been broadly studied in the last decade. Besides, a lot of interesting studies by variational methods can be found in \cite{3,4,5,7,9,10,11,12} for the nonlinear Schr\"{o}dinger equation with various growth conditions on the nonlinear term $f$.

Recently, Shi and Chen \cite{13} obtained the existence and multiplicity of nontrivial solutions for problem (\ref{1}). By using Morse theory in combination with local linking arguments, they first proved the existence of at least two nontrivial solutions for the equation (\ref{1}). Then, they obtained the existence of at least $k$ distinct pairs of solutions via Clark's theorem, when the nonlinearity $f(x,u)$ is sublinear at infinity and the potential function satisfies the following assumption:
\begin{list}{}{}
\item[$(V_1^{'})$] $V \in C(\R^N,\R)$ satisfies $\inf\limits_{x\in \R^N} V(x) \geq a_1 >0$, where $a_1$ is a constant.
\end{list}

Zhang et al. \cite{14} proved the existence of infinitely many radial and non-radial solutions for problem (\ref{1}) by means of the Symmetric Mountain Pass Theorem, when $V(x)$ is a radial function (i.e., $V(|x|)=V(x)$) and satisfies $(V_1^{'})$ and $f$ satisfies some general superlinear assumptions at infinity. Khoutir and Chen \cite{15} obtained a sequence of high energy solutions for problem (\ref{1}) by using the Symmetric Mountain Pass Theorem, when $f$ verifies a superlinear growth condition and the potential function $V$ satisfies $(V_1^{'})$ and the following assumption
\begin{list}{}{}
\item[$(V_2^{'})$] For each $M>0$, $\meas \{ x\in \R^N\::\:V(x)\leq M\}<+\infty$, where $\meas\{.\}$ denotes the Lebesgue measure in $\R^N$.
\end{list} 

Teng \cite{16} established the existence of infinitely many high or small energy solutions for problem (\ref{1}) via the variant Fountain Theorem, when $V(x)$ satisfies $(V_1^{'})$ and a weaker condition than $(V_2^{'})$, that is, 
\begin{list}{}{}
\item[$(V_2)$] There exists $d_0>0$ such that for any $M>0$,
\begin{equation*}
\meas \{ x\in \R^N\::|x-y|\leq d_0,\:V(x)\leq M\}<+\infty,
\end{equation*}
\end{list}
moreover, the nonlinear term $f(x,u)$ is assumed to be asymptotically linear or superquadratic growth.

In \cite{17} Ge improved the conclusions of Teng \cite{16}. When the potential $V(x)$ satisfies only the condition $(V_1^{'})$ and the nonlinear term $f$ satisfies some more relaxed superlinear assumptions, the author proved the existence of infinitely many solutions of problem (\ref{1}) by the aid of the variant Fountain Theorem.

In \cite{18} Du and Tian investigated the existence of infinitely many solutions of problem (\ref{1}). Firstly, the authors studied the case when $f(x,u)$ is sublinear at infinity with respect to $u$ and $V$ satisfies $(V_1^{'})$ and they obtained the existence of infinitely many small energy solutions via Dual Fountain Theorem. Then, the authors studied the existence of infinitely many high energy solution by using the Fountain Theorem, when $f$ is superlinear at infinity and the potential $V$ satisfies
\begin{list}{}{}
\item[$(V_1)$] $V \in C(\R^N,\R)$ satisfies $\inf\limits_{x\in \R^N} V(x) \geq -\infty $,
\end{list}
and $(V_2)$. 
Finally, Du and Tian proved the existence of infinitely many energy solutions when the nonlinear term is a combination of sublinear and critical terms and $V$ satisfies $(V_1^{'})-(V_2^{'})$. Note that the results of Du and Tian extend and sharply improve the results of Teng \cite{16}. 

For more interesting results on the existence and the multiplicity of solutions of problem (\ref{1}), we refer the readers to \cite{19,20,21,22,23,24,25,26,27,Amb1,Amb2,Amb3} and the references therein.

Inspired by the above papers, the aim of this paper is to establish the existence of at least one nontrivial solution and infinitely many solutions for problem (\ref{1}). Under appropriate assumptions on the nonlinear term $f(x,u)$ which is sublinear with respect to $u$, we first prove the the existence of at least one nontrivial solution for \eqref{1} via varitional methods when the potential $V$ is allowed to be sign-changing. Next we prove that the problem \eqref{1} has infinitely many solutions by means of variational methods in combination with Moser iteration method when the nonlinear term $f(x,u)$ is only locally defined for $|u|$ small and the potential $V(x)$ is a nonnegative function. Recent results from the literature are extended and improved. In order to state the main results of this paper, we make the following assumptions on $f$: 
\begin{list}{}{}
\item[$(f_1)$] There exists a constant $r\in(1,2)$ and a positive function $\xi \in L^{\frac{2}{2-r}}(\R^N)$ such that
\begin{equation*}
|f(x,u)|\leq r \xi (x)|u|^{r -1},\quad \forall (x,u) \in \R^N\times \R.
\end{equation*}
\item[$(F_1)$] There exists a constant $\delta_1>0$ such that  $f\in C(\R^N\times [-\delta_1,\delta_1],\R)$, and there exist a constant $r\in(2-\frac{4s}{N},2)$ and a positive function $\xi \in L^{\frac{2}{2-r}}(\R^N)$ such that
\begin{equation*}
|f(x,u)|\leq r \xi (x)|u|^{r -1},\quad |u| \leq \delta_1,\: \forall x \in \R^N.
\end{equation*}
\item[$(F_2)$] There exists $\delta_2>0$ such that
\begin{equation*}
f(x,u)\geq M|u|,\quad |u| \leq \delta_2,\: \forall x \in \R^N,\: \forall M>0.
\end{equation*}
\item[$(F_3)$] There exists a constant $\delta_3>0$ such that $f(x,-u)=-f(x,u)$ for all $|u|\leq \delta_3$ and all $x \in \R^N$.
\end{list}

Denote $V^{\pm}=\max\{\pm V(x),0\}$ and
\begin{equation}\label{4}
S:=\inf_{u\in H^s(\R^N)\setminus\{0\}} \frac{[u]_{H^s}^2}{\|u\|_{2_s^*}^2}.
\end{equation}
$S$ is the best constant in the Sobolev embedding $H^s(\R^N)\hookrightarrow L^{2_s^*}(\R^N)$.
Concerning the potential $V(x)$, we suppose that:
\begin{list}{}{}
\item[$(V_a)$] $V\in L_{loc}^q(\R^N)$ for some $q>\frac{N}{2s}$ and  $\lim\limits_{|x|\rightarrow \infty}V(x)=V_\infty>0$.

\item[$(V_b)$] $V^-\in L^{\frac{N}{2s}}(\R^N)$ with
\begin{equation*}
\|V^-\|_{ L^{N/2s}}<S.
\end{equation*} 
\end{list}

The main results of this paper are the following theorems.
\begin{theorem}\label{111}
Assume that $(V_a)$, $(V_b)$, $(f_1)$ and $(F_2)$ hold. Then the problem (\ref{1}) possesses at least one nontrivial solution.
\end{theorem}

\begin{theorem}\label{112}
Assume that $N>4s$ and $V(x)\geq 0$. If the conditions $(V_a)$ and $(F_1)-(F_3)$ hold, then the problem (\ref{1}) possesses infinitely many nontrivial solutions $\{u_k\}$ satisfying
\begin{equation*}
\frac{1}{2} \iint_{\R^N\times\R^N}\frac{| u_k(x)-u_k(y)|^{2}}{| x-y|^{N+2s}}dxdy+\frac{1}{2}  \int_{\R^N} V(x)u_k^2dx-\int_{\R^N}F(x,u_k)dx\leq 0
\end{equation*}
and $u_k\rightarrow 0$ as $k \rightarrow \infty$.
\end{theorem}
\begin{remark}
\begin{list}{}{}
\item[\textbf{(1)}] Unlike \cite{13,18,19}, when we prove the existence of infinitely many solutions, the nonlinear term $f$ does not satisfy any growth condition and any control at infinity, and we just require that $f(x,u)$ is a function locally odd with respect to $u$. An example of function which satisfies assumptions $(F_1)-(F_3)$ is the
following
\begin{equation*}
F(x,u)=
\begin{cases}
a(x)|u|^{r},\quad & |u|\leq 1,\\
0,\quad & |u|>1,
\end{cases}
\end{equation*}
where $r\in (2-\frac{4s}{N},2)$ and $a(x) \in L^{\frac{2}{2-r}}(\R^N,\R^+)$.
 
\item[\textbf{(2)}] Under the conditions $(V_a)-(V_b)$ it is clear that $V(x)$ does not satisfy any coerciveness condition, so the main difficulties of our problem is the lack of compactness of the Sobolev embedding. Noting that in \cite{13,18,19} the authors studied problem (\ref{1}) with a potential function $V(x)$ which is strictly positive, so our results extend and improve the aforementioned works.

\item[\textbf{(3)}] To the best of our knowledge, conditions $(V_a)-(V_b)$ was introduced by Furtado et al. in \cite{3}. Furthermore, it is not difficult to find a function $V:\R^N\mapsto \R$ satisfying the condition $(V_a)-(V_b)$, for example let
\begin{equation*}
V(x)=
\begin{cases}
\frac{|x|^2}{1+|x|^2},\quad & |x|> 1,\\
-\frac{\varepsilon}{|x|^{\alpha}},\quad & |x|\leq 1,
\end{cases}
\end{equation*}
where $\varepsilon>0$ is small and $0<\alpha<2s$, which is similar with the example appeared in \cite{3} with a slight modification. Besides, let $V(x)=\frac{|x|^2}{1+|x|^2},\: x\in \R^N$, then, it is clear that $V(x) \geq 0$ for all $x\in \R^N$ and satisfies the condition $(V_a)$, hence, one can take $V(x)$ as a potential function in Theorem \ref{112}.  
\end{list}
\end{remark}

Next, we introduce the following notations. As usual, for $1\leq p<+\infty$, we let
\begin{equation*}
\| u\|_{p}:=\left(\int_{\R^N}| u|^{p}dx\right)^{\frac{1}{p}},\quad u\in L^p(\R^N),
\end{equation*}
and
\begin{equation*}
\|u\|_{\infty}=ess\sup\limits_{x \in \R^N}|u(x)|,\quad u \in  L^{\infty}(\R^N).
\end{equation*}
Let $C_{0}^{\infty }(\R^N)$\ be the collection of smooth functions with compact support and $
\mathcal{S}(\R^N)$ the Schwartz space of rapidly decreasing $C^{\infty }$ functions in $\R^N$. We recall that the Fourier transform $\mathcal{F}\phi $ (or simply $\widehat{\phi }$) is defined for any $\phi \in \mathcal{S}(\R^N)$ as 
\begin{equation*}
\mathcal{F}\phi (\xi )=\frac{1}{(2\pi )^{N}}\int_{\R^N}e^{-\mathit{i}x\xi }\phi (x)dx.
\end{equation*}
Moreover, by Plancherel's theorem we have $\| \phi \|_{2}=\| \widehat{\phi }\| _{2}$, $\forall \phi \in \mathcal{S}(\R^N)$.
The fractional Laplacian $(-\Delta )^{s}$ with $s \in (0,1)$ of a function $\phi \in \mathcal{S}(\R^N)$ is defined by 
\begin{equation*}
\mathcal{F}((-\Delta )^{s}\phi )(\xi )=| \xi |
^{2s}\mathcal{F}\phi (\xi ),\quad \forall s \in (0,1).
\end{equation*}
If $\phi $ is sufficiently smooth, according to \cite{29}, the fractional Laplacian $(-\Delta )^{s}$ can be viewed as a pseudo-differential operator defined by
\begin{equation*}
(-\Delta )^{s}\phi (x)=C_{N,s}P.V.\int_{\R^N}\frac{\phi (x)-\phi (y)}{| x-y| ^{N+2s}}dy,
\end{equation*}
where $P.V.$ is the principal value and $C_{N,s}>0$ is a normalization
constant.
Consider the fractional Sobolev space
\begin{equation*}
H^{s} (\R^N):=\left\{{u\in L^{2}(\R^N)\: :\: \frac{| u(x)-u(y)|}{| x-y|^{\frac{N}{2}+s}}\in L^{2}(\R^N\times\R^N)}\right\}
\end{equation*}
with the inner product and the norm
\begin{equation*}
\langle u,v\rangle _{H^{s }}=\iint_{\R^N\times\R^N}\frac{(u(x)-u(y))(v(x)-v(y))}{| x-y|^{N+2s}}dxdy+\int_{\R^N}u(x)v(x)dx,
\end{equation*}
\begin{equation*}
\| u\|_{H^{s}}^{2}=\langle u,u\rangle _{H^{s}}=\iint_{\R^N\times\R^N}\frac{| u(x)-u(y)|^{2}}{| x-y|^{N+2s}}dxdy+\int_{\R^N}| u(x)|^{2} dx,
\end{equation*}
where the norm
\begin{equation*}
[u]_{H^{s}}^{2}=\iint_{\R^N\times\R^N}\frac{| u(x)-u(y)|^{2}}{| x-y|^{N+2s}}dxdy
\end{equation*}
is the so called Gagliardo semi-norm of $u$. The space $H^{s}(\R^N)$ can also be described by means of the Fourier transform. Indeed, it is defined by 
\begin{equation*}
H^{s} (\R^N):=\left\{{u\in L^{2}(\R^N)\: :\: \int_{\R^N}(| \xi | ^{2s }| \widehat{u}(\xi)| ^{2}+| \widehat{u}(\xi )| ^{2})d\xi <\infty }\right\},
\end{equation*}
endowed with the norm
\begin{equation*}
\| u\|_{H^{s }}=\left(\int_{\R^N}\left(| \xi | ^{2s}|\widehat{u}(\xi )|^{2}+|\widehat{u}(\xi )|^{2}\right)d\xi \right)^{\frac{1}{2}}.
\end{equation*}

In \cite{29}, the authors show that for $u\in \mathcal{S}(\R^N)$, one has
\begin{equation*}
2C_{N,s}^{-1}\int_{\R^N}| \xi | ^{2s}|\widehat{u}(\xi )|^{2}d\xi=2C_{N,s}^{-1}\| (-\Delta )^{\frac{s}{2}}u\|_{2}^{2}=[u]_{H^{s}}^{2}.
\end{equation*}
Therefore, the norms on $H^{s}(\R^N)$ defined below,
\begin{equation*}
\begin{split}
&u\mapsto \left( \iint_{\R^N\times\R^N}\frac{| u(x)-u(y)|^{2}}{| x-y|^{N+2s}}dxdy+\int_{\R^N}| u(x)|^{2} dx\right)^{\frac{1}{2}}\\
&u\mapsto \left(\int_{\R^N}\left(| \xi | ^{2s}|\widehat{u}(\xi )|^{2}+|\widehat{u}(\xi )|^{2}\right)d\xi \right)^{\frac{1}{2}}\\
&u\mapsto \left(\| (-\Delta )^{\frac{s}{2}}u\|_{2}^{2}+\int_{\R^N}| u(x)|^{2} dx\right)^{\frac{1}{2}}
\end{split}
\end{equation*}
are all equivalent.

The paper is organized as follows. In Section 2, we prove some lemmas, which are crucial to prove our main results. Section 3 is devoted to the proof of \autoref{111}\ and \autoref{112}.

\section{Variational framework and technical lemmas} 
In the sequel, $C,C_i>0$ denote various positive constants which may change from line to line. Let $\Omega \subset \R^N$, then, for $1\leq p\leq +\infty$, we denote by $\|.\|_{p,\Omega}$ the usual norm in $L^{p}(\Omega)$.

Let
\begin{equation*}
H=H(\R^N):=\left\{{ u \in H^s( \R^N)\::\: \int_{\R^N}V^+(x)u^{2} dx<+\infty }\right\},
\end{equation*}
Obviously, $H$ is a Hilbert space equipped with the inner product
\begin{equation*}
\langle u,v\rangle:= \langle u,v\rangle_{H} =\iint_{\R^N\times\R^N}\frac{(u(x)-u(y))(v(x)-v(y))}{| x-y|^{N+2s}}dxdy+\int_{\R^N}V^+(x)uvdx,
\end{equation*}
and the norm $\|u\|=\langle u,u\rangle^{\frac{1}{2}}$, furthermore, similar to \cite[Lemma 2.1]{3}, it is easy to see that the norm $\|.\|$ is equivalent to the usual norm of $H^s(\R^N)$.

Under the condition $(V)$ the embeddings $H(\R^N) \hookrightarrow L^p(\R^N)$ is continuous for $p \in [2,2_{s}^{\ast}]$, that is, there exist constants $\mu_p>0$ such that
\begin{equation}\label{6}
\|u\|_{p}\leq \mu_p ||u||,\quad \forall u\in H(\R^N),\: p\in [2,2_s^*],
\end{equation}
where $2_s^{*} =\frac{2N}{N-2s}$ is the critical Sobolev exponent. Moreover, from \cite[Lemma 2.1]{13}, we know that under the assumption $(V)$, the embedding $H \hookrightarrow L_{loc}^p(\R^N)$ is compact for $2\leq p<2_s^{*}$. 

For the fractional Schr\"{o}dinger equation (\ref{1}), the associated energy functional is defined on $H$ as follows 
\begin{equation}\label{7}
I(u)=\frac{1}{2}\|u\|^2-\frac{1}{2}\int_{\R^N}V^-(x)u^2dx-\int_{\R^N}F(x,u)dx.
\end{equation}
Let $0<l\leq \frac{1}{2}\min\{\delta_1,\delta_2,\delta_3\}$. We define an even function $\eta \in C^1(\R,\R)$ such that $0\leq \eta(t)\leq 1$,
\begin{equation*}
\eta(t)=
\begin{cases}
1 \quad \text{for}\quad |t|\leq l;\\
0 \quad \text{for}\quad |t|\geq 2l;
\end{cases}
\end{equation*}
and $\eta $ is decreasing in $[l,2l]$. 

Let 
\begin{equation}\label{equ0}
f_{\eta}(x,u):=\eta(u)f(x,u),\quad \forall (x,u)\in \R^N\times\R.
\end{equation}
Consider the cut-off functional $I_{\eta}$ defined by:
\begin{equation}\label{8}
I_{\eta}(u)=\frac{1}{2}\int_{\R^{2N}} \frac{| u(x)-u(y)|^{2}}{| x-y|^{N+2s}}dxdy+\frac{1}{2}\int_{\R^N} V(x)u^2dx-\int_{\R^N} F_{\eta}(x,u)dx,
\end{equation}
where $F_{\eta}(x,u)=\int_0^u \eta(s)f(x,s)ds$. Then, the critical points of $I_{\eta}$ are weak solutions of the following equation 
\begin{equation}\label{9}
(-\Delta)^s u+V(x)u=f_{\eta}(x,u), \quad  x\in \R^N.
\end{equation}
Furthermore, if $u\in H$ with $\|u\|_{\infty} \leq l$ is a critical point of the functional $I_{\eta}$, then $u$ is a weak solution of (\ref{1}).

\begin{lemma}\label{221}
Suppose that $(V_a)$, $(V_b)$ and $(F_1)$ hold. Then, the functional $I_\eta$ is well define and of class $C^1(H,\R)$ with
\begin{equation}\label{10}
\langle I_{\eta}^{\prime}(u),v\rangle= \int_{\R^{2N}}\frac{(u(x)-u(y))(v(x)-v(y))}{| x-y|^{N+2s}}dxdy+\int_{\R^N}V(x) u v dx-\int_{\R^N}f_{\eta}(x,u)v dx,
\end{equation}
for all $v\in H$. Moreover, the critical points of $I_{\eta}$ in $H$ are solutions of problem \eqref{9}.
\end{lemma}

\begin{proof}
By $(F_1)$ and (\ref{equ0}), we have
\begin{equation}\label{11}
\left|f_{\eta}(x,u)\right|\leq r \xi (x)|u|^{r-1},\quad \forall (x,u) \in \R^N\times \R,
\end{equation}
which yields 
\begin{equation*}
\begin{split}
\left|F_{\eta}(x,u)\right|=&\left|F_{\eta}(x,u)-F_{\eta}(x,0)\right|\\
& \leq \int_0^1 |f_{\eta}(x,tu)||u|dt\\
&\leq \xi (x)|u|^{r},\quad \forall (x,u)\in \R^N\times \R,
\end{split}
\end{equation*}
where $r\in (2-\frac{4s}{N},2)$ and $\xi \in L^{\frac{2}{2-r}}(\R^N)$. Then we obtain
\begin{equation}\label{12}
\begin{split}
\int_{\R^N}F_{\eta}(x,u)dx & \leq \int_{\R^N} \xi(x)|u|^{r}dx\\
&\leq \left(\int_{\R^N}|\xi(x)|^{\frac{2}{2-r}}dx\right)^{\frac{2-r}{2}}\left(\int_{\R^N}|u|^2dx\right)^{\frac{r}{2}}\\
& \leq  \|\xi(x)\|_{\frac{2}{2-r}}\|u\|_{2}^{r}\\
& \leq \mu_{2}^{r} \|\xi(x)\|_{\frac{2}{2-r}}\|u\|^{r}.
\end{split}
\end{equation} 
On the other hand, since $V^-\in L^{\frac{N}{2s}}(\R^N)$, by (\ref{4}) and the H\"{o}lder inequality we have
\begin{equation}\label{5}
\int_{\R^N}V^{-}(x)u^2dx \leq \|V^-\|_{\frac{N}{2s}}\|u\|_{2_s^*}^2\leq S^{-1}\|V^-\|_{\frac{N}{2s}}[u]_s^2\leq S^{-1}\|V^-\|_{\frac{N}{2s}}\|u\|^2, 
\end{equation}
for all $u\in H$.
Hence, $I_{\eta}$ is well defined on $H$. Next, we prove that (\ref{10}) holds. According to (\ref{8}), it suffices to show that 
\begin{equation*}
\Psi \in C^{1}(H,\R),\quad \langle\Psi^{\prime}(u),v \rangle= \int_{\R^N} f_{\eta}(x,u)vdx,\quad \forall u,v\in H,
\end{equation*} 
where $\Psi (u)=\int_{\R^N} F_{\eta}(x,u)dx$.

For any function $\theta : \R^N\rightarrow (0,1)$ and $t \in (0,1)$, by $(F_{1})$, (\ref{4}) and the H\"{o}lder inequality, one has
\begin{equation}\label{13}
\begin{split}
&\int_{\R^N}\max\limits_{t\in (0,1)} \left| f_{\eta}\left(x,u+t\theta(x)v\right) v\right| dx\\
&\leq \int_{|u +t\theta (x)v| \leq 2l} r \xi(x) |u +t\theta (x)v|^{r-1} |v|dx\\
&\leq  \int_{|u +t\theta (x)v| \leq 2l} r\xi(x)\left(|u|^{r-1} +|v|^{r-1}\right)|v|dx\\
&\leq  r \|\xi(x)\|_{\frac{2}{2-r}}\left(\|u\|_{2}^{r-1}\|v\|_{2}+\|v\|_{2}^{r}\right)\\
&\leq  r \mu_2^{r}\|\xi(x)\|_{\frac{2}{2-r}} \left(\|u\|^{r-1} \|v\|+ \|v\|^{r}\right).
\end{split}
\end{equation}
Then by (\ref{13}) and the Lebesgue's Dominated Convergence Theorem, we have
\begin{equation} \label{equ 01}
\begin{split}
\langle \Psi^{\prime}(u),v\rangle = & \lim\limits_{t\rightarrow 0^{+}}\frac{\Psi(u+tv)-\Psi(u)}{t}\\
= & \lim\limits_{t\rightarrow 0^{+}}\int_{\R^N}\frac{F_{\eta}(x,u+tv)-F_{\eta}(x,u)}{t}dx\\
= & \lim\limits_{t\rightarrow 0^{+}}\int_{\R^N} f_{\eta}(x,u+t\theta(x)v) v(x)dx\\
= & \int_{\R^N} f_{\eta}(x,u) vdx,
\end{split}
\end{equation}
which implies that (\ref{10}) holds. Moreover, by a standard argument, it is easy to show that the critical points of $I_{\eta}$ are solutions of problem (\ref{1}). It remains to show that $ \Psi^{\prime}$ is continuous. Let $\{u_n\}\subset H$ be a sequence such that $u_n\rightarrow u$ in $H$, therefore $u_n \rightarrow u$ in $L^2(\R^N)$ and 
\begin{equation}\label{equ 1}
\lim_{n\rightarrow\infty}u_n(x) = u(x),\quad \text{a.e. } x\in \R^N.
\end{equation}
We claim that
\begin{equation}\label{equ 2}
f_{\eta}(x,u_n)\rightarrow f_{\eta}(x,u) \text{ strongly in } L^2(\R^N).
\end{equation}
Arguing by contradiction, there exist $\varepsilon_0>0$ and a subsequence $\{u_{n_k}\}$ such that 
\begin{equation}\label{equ 3}
\int_{\R^N} |f_{\eta}(x,u_{n_k})-f_{\eta}(x,u)|^2 dx\geq \varepsilon_0,\quad \forall k\in \N.
\end{equation}
Since $u_n\rightarrow u$ in $L^2(\R^N)$, passing to a subsequence if necessary, one can assume that 
\begin{equation*}
\sum_{k=1}^{\infty} |u_{n_k}-u|_{2,\R^N}^2<+\infty.
\end{equation*}
Therefore, $g(x):=\left[\sum_{k=1}^{\infty}|u_{n_k}-u|^2\right]^{1/2} \in L^2(\R^N)$. On the other hand, by (\ref{11}) and H\"{o}lder's inequality we have
\begin{equation}\label{equ4}
\begin{split}
|f_{\eta}(x,u_{n_k})-f_{\eta}(x,u)|^2 & \leq 2|f_{\eta}(x,u_{n_k})|^2+2|f_{\eta}(x,u)|^2\\
&\leq 4r^2 |\xi(x)|^2 \left(|u_{n_k}(x)|^{2(r-1)}+|u(x)|^{2(r-1)}\right)\\
&\leq C_0 |\xi(x)|^2 \left(|g(x)|^{2(r-1)}+|u(x)|^{2(r-1)}\right)\\
&:=w(x),\quad \forall k\in \N, \: x\in \R^N
\end{split}
\end{equation}
and
\begin{equation}\label{equ5}
\begin{split}
\int_{\R^N}w(x)dx&= C_0\int_{\R^N}|\xi(x)|^2 \left(|g(x)|^{2(r-1)}+|u(x)|^{2(r-1)}\right)dx\\
&\leq C_0 \|\xi\|_{\frac{2}{2-r}}^2\left(\|g\|_{2}^{2(r-1)}+\|u\|_{2}^{2(r-1)}\right)<+\infty.
\end{split}
\end{equation}
Combining (\ref{equ 1}), (\ref{equ4}), (\ref{equ5}) with Lebesgue's Dominated Convergent Theorem we conclude
\begin{equation*}
\lim_{k\rightarrow \infty}\int_{\R^N} |f_{\eta}(x,u_{n_k})-f_{\eta}(x,u)|^2 dx=0,
\end{equation*}
which contradicts (\ref{equ 3}). Thus, (\ref{equ 2}) holds.

It follows from (\ref{equ 01}), (\ref{equ 2}) and H\"older's inequality that
\begin{equation*}
\begin{split}
\left|\langle \Psi'(u_n)-\Psi'(u),v\rangle\right|&\leq \int_{\R^N} |f_{\eta}(x,u_{n})-f_{\eta}(x,u)||v| dx\\
&\leq \left(\int_{\R^N} |f_{\eta}(x,u_{n})-f_{\eta}(x,u)|^2 dx\right)^{1/2}\|v\|_{2}\\
& \leq \mu_2 \left(\int_{\R^N} |f_{\eta}(x,u_{n})-f_{\eta}(x,u)|^2 dx\right)^{1/2}\|v\|\rightarrow 0,
\end{split}
\end{equation*}
as $n\rightarrow \infty$, which shows that $\Psi'$ is continuous. The proof is completed.
\end{proof}

\begin{corollary}\label{2221}
Suppose that $(V_a)$, $(V_b)$ and $(f_1)$ hold. Then, the functional $I$ is well define and of class $C^1(H,\R)$ with
\begin{equation*}
\langle I^{\prime}(u),v\rangle= \int_{\R^{2N}}\frac{(u(x)-u(y))(v(x)-v(y))}{| x-y|^{N+2s}}dxdy+\int_{\R^N}V(x) u v dx-\int_{\R^N}f(x,u)v dx,
\end{equation*}
for all $v\in H$. Moreover, the critical points of $I$ in $H$ are solutions of problem (\ref{1}).
\end{corollary}

Recall that a sequence $\{u_{n}\}\subset E$ is said to be a Palais-Smale sequence at the level $c\in \R$ ((PS)$_{c}$ sequence for short) if $I(u_{n})\rightarrow c$ and $I^{\prime}(u_{n})\rightarrow 0$, $I$ is said to satisfy the Palais-Smale condition at the level $c$ ((PS)$_{c}$ condition for short) if any (PS)$_{c}$-sequence has a convergent subsequence.

\begin{lemma}\label{222}\cite{28}
Let $E$ be a Banach space and $I\in C^1(E,\R)$ satisfy the (PS) condition. If $I$ is bounded from below, then $c=\inf_{E} I$ is a critical value of $I$.
\end{lemma}

In order to find the multiplicity of nontrivial critical points of $I$, we will use the genus properties, so we recall the following definitions and results (see \cite{30}).

Let $E$ be a Banach space and $I\in C^1(E,\R)$. We set 
\begin{equation*}
\Gamma =\{A\subset E-{0}\: : \: A \: \text{is closed in}\: E \: \text{and symmetric with respect to}\: 0\}.  
\end{equation*} 
\begin{definition}
For $A\in \Gamma$, we say genus of $A$ is $k$ denoted by $\gamma(A)=k$ if there is an odd map $\Psi \in C(A, \R^N \setminus 0)$ and $k$ is the smallest integer with this property. 
\end{definition}   
For any $k \in \N$, we set 
\begin{equation*}
\Gamma_k=\{ A\in \Gamma \: : \: \gamma(A)\geq k\}.
\end{equation*}
Then, we have the following lemma from \cite{30}.
\begin{lemma}\label{223}
Let $E$ be an infinite dimensional Banach space and $I \in C^1(E,\R)$ satisfies $(A_1)$ and $(A_2)$ below:
\begin{list}{}{}
\item[$(A_1)$] $I$ is even, bounded from below, $I(0)=0$ and $I$ satisfies the (PS) condition.
\item[$(A_2)$] For each $k\in \N$, there exists an $A_k\in \Gamma_k$ such that $\sup_{u\in A_k} I(u)<0$.
\end{list}
Then $I$ admits a sequence of critical points $u_k$ such that $I(u_k)\leq 0$, $u_k\neq 0$ and $\lim_{k\rightarrow \infty}u_k=0$.
\end{lemma}
\section{Proof of the main results}
\begin{lemma}\label{331}
Under the assumptions $(V_a)$, $(V_b)$ and $(F_1)$, the functional $I_{\eta}$ is bounded from below and satisfies the (PS) condition.
\end{lemma}
\begin{proof}
By $(V_b)$, (\ref{8}), (\ref{12}) and (\ref{5}) we have
\begin{equation}\label{14}
\begin{split}
I_{\eta}(u)=&\frac{1}{2}\|u\|^2-\frac{1}{2}\int_{\R^N} V^-(x)u^2dx-\int_{\R^N}F_{\eta}(x,u)dx\\
\geq & \frac{1}{2} \|u\|^2-\frac{1}{2} S^{-1}\|V^-\|_{\frac{N}{2s}}\|u\|^2-\mu_{2}^{r} \|\xi(x)\|_{\frac{2}{2-r}} \|u\|^{r}\\
=& \frac{1}{2}\left(1-S^{-1}\|V^-\|_{\frac{N}{2s}}\right)\|u\|^2-\mu_{2}^{r} \|\xi(x)\|_{\frac{2}{2-r}}\|u\|^{r}.
\end{split}
\end{equation} 
Then by (\ref{14}) we conclude that $I_{\eta}$ is bounded from below since $r\in (2-\frac{4s}{N},2)$.

Next, we prove that $I_{\eta}$ satisfies the (PS) condition. Let $\{u_n\}\subset H$ be any (PS) sequence of $I_{\eta}$, i.e., $\{I_{\eta}(u_n)\}$ is bounded and $I^{'}_{\eta}(u_n)\rightarrow 0$ in $H^*$.

From (\ref{14}) we have
\begin{equation*}
C_1\geq I_{\eta}(u_n)\geq \frac{1}{2}\left(1-S^{-1}\|V^-\|_{\frac{N}{2s}}\right)\|u_n\|^2-\mu_{2}^{r} \|\xi(x)\|_{\frac{2}{2-r}}\|u_n\|^{r}.
\end{equation*}
This implies that $\{u_n\}$ is bounded in $H$ since $r\in (1,2)$ and $C_1$ is independent of $n$, that is, there exists a constant $C_2>0$ which is independent of $n$ such that
\begin{equation}\label{15}
\|u_n\| \leq C_2,\quad \forall n\in \N.
\end{equation}
Therefore, up to a subsequence, there exists $u \in H$ such that $u_n \rightharpoonup u$ in $H$ and
\begin{equation}\label{16}
u_n\rightarrow u\text{ in } L_{loc}^p(\R^N),\quad p\in [2,2_s^*).
\end{equation}
By $(F_1)$, for any given $\varepsilon>0$, we can choose $R>0$ such that 
\begin{equation}\label{17}
\left(\int_{|x|>R}|\xi(x)|^{\frac{2}{2-r}}dx\right)^{\frac{2-r}{2}}<\varepsilon.
\end{equation}
On the other hand, from (\ref{16}) we get
\begin{equation*}
\lim_{n\rightarrow\infty}\int_{|x|\leq R}|u_n-u|^2dx=0,
\end{equation*}
which implies that there exists $n_0\in \N$ such that
\begin{equation}\label{19}
\int_{|x|\leq R}|u_n-u|^2dx\leq \varepsilon^2,\quad \text{for all } n\geq n_0.
\end{equation}
Therefore, for any $n\geq n_0$, exploiting $(F_1)$, (\ref{15}), (\ref{19}) and H\"{o}lder's inequality we can infer that
\begin{equation}\label{20}
\begin{split}
& \int_{B_R} |f_{\eta}(x,u_{n})-f_{\eta}(x,u)||u_n-u| dx\\
&\leq \left(\int_{B_R} |f_{\eta}(x,u_{n})-f_{\eta}(x,u)|^2 dx\right)^{1/2}\|u_n-u\|_{2,B_R}\\
& \leq \varepsilon \left[\int_{B_R} 2\left(|f_{\eta}(x,u_{n})|+|f_{\eta}(x,u)|^2\right) dx\right]^{1/2}\\
&\leq 2\varepsilon \left[ r^2 \int_{B_R}|\xi(x)|^2 \left(|u_{n}(x)|^{2(r-1)}+|u(x)|^{2(r-1)}\right)dx\right]^{1/2}\\
&\leq  2\varepsilon \left[ r^2 \|\xi\|_{\frac{2}{2-r},B_R}^2\left(\|u_n\|_{2,B_R}^{2(r-1)}+\|u\|_{2,B_R}^{2(r-1)}\right)\right]^{1/2}\\
&\leq 2\varepsilon \left[  r^2 \|\xi\|_{\frac{2}{2-r},B_R}^2\left(C_2^{2(r-1)}+\|u\|_{2,B_R}^{2(r-1)}\right)\right]^{1/2}\\
&\leq \varepsilon C_3,
\end{split}
\end{equation}
where $B_R:=\{ x\in \R^N\: :\: |x|<R\}$. Let $\Omega_R:=\R^N\setminus B_R$, then combining $(F_1)$, (\ref{15}), (\ref{17}) with the H\"{o}lder inequality, one has
\begin{equation}\label{21}
\begin{split}
& \int_{\Omega_R} |f_{\eta}(x,u_{n})-f_{\eta}(x,u)||u_n-u| dx\\
&\leq  r \int_{\Omega_R}|\xi(x)| \left(|u_{n}(x)|^{(r-1)}+|u(x)|^{(r-1)}\right)\left(|u_{n}(x)|+|u(x)|\right)dx\\
& \leq 2  r \int_{\Omega_R}|\xi(x)| \left(|u_{n}(x)|^{r}+|u(x)|^{r}\right)dx\\
&\leq 2r\left(\int_{\Omega_R}|\xi(x)|^{\frac{2}{2-r}}dx\right)^{\frac{2-r}{2}}\left(\|u_n\|_{2,\Omega_R}^{r}+\|u\|_{2,\Omega_R}^{r}\right)\\
&\leq \varepsilon C_4,
\end{split}
\end{equation}
this together with (\ref{20}) implies that 
\begin{equation}\label{22}
 \int_{\R^N} |f_{\eta}(x,u_{n})-f_{\eta}(x,u)||u_n-u| dx\rightarrow 0,
\end{equation}
as $n\rightarrow \infty$. Since $I'_\eta(u_n)\rightarrow 0$ in $H^*$, it follows from $(V)$, (\ref{10}) and (\ref{5}) that
\begin{equation*}
\begin{split}
o_n(1)& =\langle I_{\eta}^{'}(u_n)-I_{\eta}^{'}(u),u_n-u\rangle\\
&= \|u_n-u\|^2-\int_{\R^N}V^-(x)|u_n-u|^2dx-\int_{\R^N}\left[f_{\eta}(x,u_n)-f_{\eta}(x,u)\right](u_n-u) dx\\
&\geq \left(1-S^{-1}\|V^-\|_{\frac{N}{2s}}\right)\|u_n-u\|^2-\int_{\R^N}\left[f_{\eta}(x,u_n)-f_{\eta}(x,u)\right](u_n-u) dx,
\end{split}
\end{equation*}
and then
\begin{equation}\label{23}
\|u_n-u\|^2  \leq o_n(1)+\frac{1}{1-S^{-1}\|V^-\|_{N/2s}} \int_{\R^N}\left[f_{\eta}(x,u_n)-f_{\eta}(x,u)\right](u_n-u)dx.
\end{equation} 
Consequently, by (\ref{22}) and (\ref{23}) we conclude that
\begin{equation*}
u_n\rightarrow u, \quad \text{strongly in} \: H\quad \text{as}\quad n\rightarrow \infty.
\end{equation*}
Thus, the proof is completed.
\end{proof}

\begin{corollary}\label{3331}
Assume that $(V_a)$, $(V_b)$ and $(f_1)$, the functional $I$ is bounded from below and satisfies the (PS) condition.
\end{corollary}

\begin{proof}[\textbf{Proof of Theorem 1.1}]
By corollary \ref{3331}, $I$ is bounded from below and satisfies the (PS) condition. Then Lemma \ref{222} implies that $c=\inf_{H} I(u)$ is a critical value of $I$, that is, there exists a critical point $u^*\in H$ of $I$ such that $I(u^*)=c$. Next, we show that $u^*\neq 0$. Let $u_0 \in C_0^\infty(\R^N)\setminus \{0\}$ and $\Omega=\{x \in \R^N\: ;\: |u_0(x)|\leq 1\}$, then by $(F_2)$ we have that there exists a small $\delta_2 >0$ such that 
\begin{equation}\label{24}
F(x,u) \geq M|u|^2,\quad  |u| \leq \delta_2,\: \forall x\in \R^N, \: \forall M>0.
\end{equation} 
For $0<t<\delta_2$, it follows from (\ref{7}) and (\ref{24}) that
\begin{equation}\label{25}
\begin{split}
I(tu_0) & =  \frac{1}{2}t^2  \|u_0\|^{2}-\frac{t^2}{2}\int_{\R^N}V^-(x)u_0^2dx-\int_{\R^N}F(x,tu_0)dx\\
& \leq \frac{1}{2}t^2  \|u_0\|^{2} -t^2 M\int_{\Omega}|u_0|^2dx.
\end{split}
\end{equation}
Choosing $M>0$ sufficiently large such that $2M\int_{\Omega}|u_0|^2dx-\|u_0\|^{2}>0$, it follows from (\ref{25}) that $I(tu_0)<0$ for $t>0$ small enough. Hence $I(u^*)=c<0$, which implies that $u^*$ is a nontrivial critical point of $I$, and thus $u^*$ is a nontrivial solution of problem (\ref{1}). The proof is completed.
\end{proof}

\begin{lemma}\label{332}
Assume that $(V_a)$ and $(F_2)$ hold. Then, for any $n\in \N$, there exists a closed symmetric subset $A_n\subset H$ such that the genus $\gamma(A_n)\geq n$ and $\sup_{u\in A_n}I_{\eta}(u)<0$.
\end{lemma}

\begin{proof}
Let $H_n$ be any $n-$dimensional subspace of $H$. Since all norms are equivalent in a finite dimensional space, there is a constant $\beta =\beta(H_n)$ such that
\begin{equation}\label{26}
\|u\| \leq \beta |u|_{2},
\end{equation}
for all $u\in H_{n}$.
\item[\textbf{Claim.}] There exists a constant $\tau>0$ such that 
\begin{equation}\label{27}
\frac{1}{2} \int_{\R^N}|u|^2dx\geq \int_{|u|>l}|u|^2 dx
\end{equation}
for all $u\in H_n$ with $\|u\| \leq \tau$. In fact, if (\ref{27}) is false, there exists a sequence $\{u_k\}\in H_n$ such that $u_k\rightarrow 0$ in $H$ and 
\begin{equation*}
\frac{1}{2} \int_{\R^N}|u_k|^2dx< \int_{|u_k|>l}|u_k|^2 dx,\quad k\in \N.
\end{equation*}
Let $w_k :=\frac{u_k}{\|u_k\|_{2}}$. Then, we obtain
\begin{equation}\label{28}
\frac{1}{2}<\int_{|u_k|>l}|w_k|^2 dx,\quad k\in \N.
\end{equation}
On the other hand, we can assume that $w_k \rightarrow w$ in $H$ since $H_n$ is finite dimensional. Hence $w_k\rightarrow w$ in $L^2(\R^N)$. Moreover, it can be deduced from $u_k\rightarrow 0$ in $H$ that
\begin{equation*}
\meas\{x\in \R^N\: :\: |u_k|>l\}\rightarrow 0, \quad k\rightarrow \infty.
\end{equation*}
Therefore, 
\begin{equation*}
\int_{|u_k|>l}|w_k|^2 dx\leq 2\int_{|u_k|>l}|w_k-w|^2 dx+2\int_{|u_k|>l}|w|^2 dx\rightarrow 0,\quad k\rightarrow \infty,
\end{equation*}
which contradicts (\ref{28}). Thus, (\ref{27}) holds.

By $(F_2)$ we have
\begin{equation*}
f_{\eta}(x,u)\geq 4 \beta^2 u,\quad  |u| \leq 2l, \: \forall x\in \R^N.
\end{equation*}
This inequality implies that 
\begin{equation}\label{29}
F_{\eta}(x,u)=F(x,u)\geq 2\beta^2 u^2,\quad \forall (x,u)\in \R^N\times \R,\quad |u|\leq l.
\end{equation}
Therefore, it follows from (\ref{8}), (\ref{27}) and (\ref{29}) that
\begin{equation*}
\begin{split}
I_{\eta}(u) & = \frac{1}{2} \|u\|^{2}-\frac{1}{2}\int_{\R^N}V^-(x)u^2dx-\int_{\R^N} F_{\eta}(x,u)dx\\
& \leq \frac{1}{2} \|u\|^{2} -\int_{|u|\leq l} F_{\eta}(x,|u|)dx\\ 
& \leq \frac{1}{2} \|u\|^{2} -2\beta^2\int_{|u|\leq l}|u|^2dx\\
& \leq \frac{1}{2}\|u\|^{2}-2\beta^2\left( \int_{\R^N}|u|^2dx-\int_{|u|> l}|u|^2dx\right)\\
& \leq -\frac{1}{2}\|u\|^2,
\end{split}
\end{equation*}
for all $u\in H_n$ with $\|u\| \leq \min\{\tau,1\}$.

Let $0<\rho\leq \min\{\tau,1\}$ and $A_n=\{u\in H_n\::\:\|u\|=\rho\}$. We conclude that $\gamma(A_n)\geq n$ and $\sup_{u\in A_n} I_{\eta}(u)\leq -\frac{1}{2}\|u\|^2<0$. The proof is completed.
\end{proof}

In order to prove the conclusion of Theorem \ref{112}, we shall use the Moser iteration technique (see \cite{Amb4,Lin,Moser}) to show that the sequence of solutions tends to zero in  $L^\infty(\R^N)$. First, we recall some preliminary concepts about the $s-$harmonic extension problem related to the problem \eqref{9} (see \cite{cafar}). For $s\in (0,1)$, we define the space $X^s$ as the completion of $C_0^\infty(\R_+^{N+1})$ with respect to the norm
\begin{equation*}
\|v\|_{X^s}^2=\frac{1}{\kappa_s}\int_{\R_+^{N+1}}y^{1-2s}|\nabla v|^2 dxdy,
\end{equation*}
where $\kappa_s=(2^{1-2s}\Gamma(1-s))/\Gamma(s)$ and $\Gamma$ is the well known gamma function. According to \cite{cafar}, $u$ is a weak solution for \eqref{9}, if and only if $v=E_s(u)$ is a weak solution for the problem
\begin{equation}\label{30}
\begin{cases}
& div(y^{1-2s} \nabla v)=0,\quad \text{in} \: \R_+^{N+1},\\
&-\frac{1}{\kappa_s}\lim\limits_{y\rightarrow 0^+} \frac{\partial v}{\partial y}(x,y)=f_\eta(x,v(x,0))-V(x)v(x,0)\quad \text{in} \: \R^{N}.
\end{cases}
\end{equation}
Here $E_s(u)$ is the $s-$harmonic extension of $u$, the unique solution of the minimization problem
\begin{equation*}
\min\{\frac{1}{\kappa_s}\int_{\R_+^{N+1}}y^{1-2s}|\nabla v|^2 dxdy,\: v\in X^s \text{ and } v(.,0)=u\: \text{ on } \R^N\}.
\end{equation*}
Furthermore, by \cite{brand}, $E_s$ is an isometry, that is, $\|E_s(u)\|_{X^s}=\|u\|_{H^s(\R^N)}$.
\begin{lemma}\label{333}
Under the assumptions of \autoref{112}, if $\{u_k\}$ is a critical point sequence of $I_\eta$ satisfying $u_k \rightarrow 0$ in $H$ as $k\rightarrow \infty$, then $u_k \rightarrow 0$ in $L^\infty(\R^N)$ as $k\rightarrow \infty$.
\end{lemma}

\begin{proof}
First, recall that $v$ is a weak solution of \eqref{30}, if $v$ satisfies the identity
\begin{equation}\label{31}
\frac{1}{\kappa_s}\int_{\R_+^{N+1}}y^{1-2s} \nabla v \nabla \varphi dxdy + \int_{\R^N} V(x)v(x,0)\varphi(x,0)dx=\int_{\R^N} f_\eta(x,v(x,0))\varphi(x,0)dx,\quad \forall \varphi\in X^s.
\end{equation}
Hereafter, we assume that $\kappa_s=1$ for simplicity. Let $u$ be a critical point of $I_\eta$ and $v$ its $s-$harmonic extension. For any $L>0$ we define
\begin{equation}\label{32}
v_L:=
\begin{cases}
v(x),\quad & \text{if } |v(x)|\leq L,\\
 L,\quad & \text{if } v(x)\geq L.
\end{cases}
\end{equation}
For $\beta\geq 0$, taking $\varphi=v_L^{2\beta}v \in X^s$ as a test function in \eqref{31}, one has
\begin{equation}\label{33}
\begin{split}
\int_{\R_+^{N+1}}y^{1-2s} v_L^{2\beta} |\nabla v|^2 dxdy +2\beta \int_{\{v\leq L\}}y^{1-2s} v_L^{2\beta} |\nabla v|^2 dxdy=&\int_{\R^N} f_\eta(x,v(x,0))v_L^{2\beta}(x,0)v(x,0)dx\\
&- \int_{\R^N} V(x)v^2(x,0)v_L^{2\beta}(x,0)dx.
\end{split}
\end{equation}
Since $V(x)\geq 0$, it follows from \eqref{11} and \eqref{33} that
\begin{equation}\label{34}
\begin{split}
\int_{\R_+^{N+1}}y^{1-2s} v_L^{2\beta} |\nabla v|^2 dxdy &\leq r\int_{\R^N}\xi(x) v^r(x,0) v_L^{2\beta}(x,0)dx\\
& \leq r \|\xi\|_{\frac{2}{2-r}} \|v^r(.,0) v_L^{2\beta}(.,0)\|_{\frac{2}{r}}.
\end{split}
\end{equation}
Let $\overline{v}_L=v_L^\beta v$, following \cite[Lemma 4.1]{Amb4}, we have
\begin{equation}\label{35}
\|\overline{v}_L(.,0)\|_{2_s^*}^2 \leq 4S(1+\beta)^2\int_{\R_+^{N+1}}y^{1-2s} v_L^{2\beta} |\nabla v|^2 dxdy.
\end{equation}
Putting together \eqref{34} and \eqref{35}, we conclude that
\begin{equation}\label{36}
\begin{split}
\|\overline{v}_L(.,0)\|_{2_s^*}^2 &\leq 4S(1+\beta)^2 r \|\xi\|_{\frac{2}{2-r}} \|v^r(.,0) v_L^{2\beta}(.,0)\|_{\frac{2}{r}}\\
& \leq c_0^2 (1+\beta)^2 \|v^r(.,0) v_L^{2\beta}(.,0)\|_{\frac{2}{r}},
\end{split}
\end{equation}
where $c_0^2=\max\{1, 4Sr \|\xi\|_{\frac{2}{2-r}}\}$. Passing to the limit in \eqref{36} as $L\rightarrow \infty$, Fatou's Lemma yields
\begin{equation}\label{37}
\|u\|_{(1+\beta)2_s^*}\leq [c_0(1+\beta)]^{\frac{1}{1+\beta}}\|u\|_{r^*}^{\frac{r+2\beta}{2(1+\beta)}},
\end{equation}
where $r^*=\frac{2(r+2\beta)}{r}$. Following \cite[Lemma 3.4]{Lin}, set $\beta_0=0$ and $(1+\beta_{n-1})2_s^*=\frac{2(r+2\beta_n)}{r}$ for all $n\in \N$, therefore, it is easy to see that
\begin{equation}\label{38}
\beta_n=\frac{(2_s^*-2)r}{2_s^*r-4}\left(\overline{r}^n -1\right), \forall n\in \N,
\end{equation}
where $\overline{r}=\frac{2_s^*}{4}r>1$ since $r>2-\frac{4s}{N}$. For each $n\in \N$, set
\begin{equation*}
\zeta_n =\sum_{i=0}^{n-1} \frac{\ln(c_0(\beta_i+1))}{\beta_i+1}\: \text{ and } \: \sigma_n=\prod_{i=0}^{n-1}\frac{r+2\beta_i}{2(1+\beta_i)}. 
\end{equation*}
Then, according to \cite{Lin}, both $\{\zeta_n\}$ and $\{\sigma_n\}$ are convergent sequences with $\lim\limits_{n\rightarrow \infty} \zeta_n = \zeta>0$ and $\lim\limits_{n\rightarrow \infty} \sigma_n = \sigma\in (0,1]$. By iterating \eqref{37} we obtain
\begin{equation}\label{39}
\|u\|_{(1+\beta_n)2_s^*}\leq e^{\zeta_n}\|u\|_{2}^{\sigma_n },\quad \forall n\in \N.
\end{equation}
Passing to the limit as $n\rightarrow \infty$ in \eqref{39}, we get
\begin{equation*}
\|u\|_{\infty}\leq e^{\zeta}\|u\|_{2}^{\sigma },
\end{equation*}
this together with \eqref{6} implies that if $\{u_k\}$ is a critical point sequence of $I_\eta$ satisfying $u_k \rightarrow 0$ in $H$ as $k\rightarrow \infty$, then $u_k \rightarrow 0$ in $L^\infty(\R^N)$ as $k\rightarrow \infty$.
\end{proof}

\begin{proof}[\textbf{Proof of Theorem 1.2}]
By $(F_1)$ and $(F_3)$, we get that $I_{\eta}$ is even and $I_{\eta}(0)=0$. On the other hand, by Lemmas \ref{331} and \ref{332} all the conditions of Lemma \ref{223} are satisfied, which implies that $I_{\eta}$ has a sequence of critical points $\{u_k\}$ converging to $0$ in $H$ as $k\rightarrow \infty$. Therefore, $\{u_k\}$ are solutions to the problem \eqref{9}. By Lemma \ref{333}, we know that $u_k \rightarrow 0$ in $L^\infty(\R^N)$ as $k\rightarrow \infty$. Hence, there exists $k_0\in \N$ such that $\|u_k\|_{\infty}\leq l$ for each $k\geq k_0$. Thus, we get infinitely many solutions of \eqref{1}. This completes the proof.
\end{proof}

\section*{Acknowledgments}
The author would like to express sincere thanks to the anonymous referee for the valuable comments and suggestions which improved the final version of the manuscript.

\end{document}